\documentclass[10pt]{article}
\usepackage{amsmath}
\usepackage{amssymb}
\usepackage{amsfonts}
\usepackage{amsthm}
\usepackage{amscd}
\usepackage{mathdots}
\usepackage{verbatim}
\usepackage[english]{babel}
\usepackage[utf8x]{inputenc}
\usepackage{amsfonts}
\usepackage[all,cmtip]{xy}
\usepackage[active]{srcltx}
\usepackage{color}
\usepackage{enumerate}

\makeatletter
\def\revddots{\mathinner{\mkern1mu\raise\p@
\vbox{\kern7\p@\hbox{.}}\mkern2mu
\raise4\p@\hbox{.}\mkern1mu\raise7\p@\hbox{.}\mkern1mu}}
\makeatother

\theoremstyle{plain}
\newtheorem{thm}{Theorem}[section]
\newtheorem{lem}[thm]{Lemma}
\newtheorem{prop}[thm]{Proposition}

\theoremstyle{definition}
\newtheorem{defn}{Definition}[section]

\theoremstyle{remark}
\newtheorem*{rem}{Remark}

\newcommand{\C}{\mathbb{C}}

\newcommand{\Ind}{\mathrm{Ind}}
\newcommand{\ind}{\mathrm{ind}}

\newcommand{\deri}{\mathrm{d}}

\def\GL{\operatorname{GL}}

\def\Sp{\operatorname{Sp}}
\def\O{\operatorname{O}}
\def\SO{\operatorname{SO}}

\def\span{\operatorname{span}}

\makeatletter
\def\Ddots{\mathinner{\mkern1mu\raise\p@
\vbox{\kern7\p@\hbox{.}}\mkern2mu
\raise4\p@\hbox{.}\mkern2mu\raise7\p@\hbox{.}\mkern1mu}}
\makeatother

\title{Generalized Shalika model on $\SO_{4n}(F),$  symplectic linear model on $\Sp_{4n}(F)$ and theta correspondence}

\author{Marcela Hanzer}
\date{}

\begin{document}
\maketitle

\begin{abstract} We show that if an irreducible admissible representation of $SO_{4n}(F)$ has a generalized Shalika model, then its small theta lift to $\Sp_{4n}(F)$ has the symplectic linear model, thus answering a question posed by D. Jiang. Here $F$ is a non-archimedean field of characteristic zero.
\end{abstract}

\section{Introduction} 
The fundamental results of Arthur led to the classification of the automorphic discrete spectrum of  the classical groups. The automorphic representations of a classical group are grouped into global (Arthur) packets. Global  Arthur packets are formed using local Arthur packets. It is very important to have a way  to distinguish representations inside a local  packet, being it Arthur or Langlands packet. The characterization of the representations in a packet by models they have turns out to be very important; let us just mention the landmark work of Gan, Gross, Prasad, Waldspurger and others on restriction problems for classical groups and  existence of Bessel and Fourier-Jacobi models (\cite{GGP_I},\cite{GGPW}, etc.). The second use of models for groups over local fields or periods for groups over global fields is their use for the determination of poles of the global L-functions. In that way D. Jiang introduced the generalized Shalika model for the split group $\SO_{4n}(F),$ where $F$ is a local non-archimedean field of characteristic zero. In more detail, Jiang introduced this model in \cite{JQ_residues} with the Langlands-Shahidi method to characterize irreducible automorphic cuspidal representations $\pi$ of $\GL_{2n}$ whose global L--function $L(s,\pi,\Lambda^2)$ has a pole for $s=1.$
Moreover, Jiang  formulated conjectures about the characterizations of local Arthur packets containing a member having a non-zero generalized Shalika model (cf. the fourth section of \cite{JNQ}); these conjectures can be viewed as a specific case of on-going research into spherical varieties (cf. \cite{Sakellaridis_spherical}).

Jiang also observed the following: let $\sigma$ be an irreducible cuspidal  symplectic representation of $\GL_{2n}(F),$ where $F$ is a non-archimedean fields of characteristic zero. If we induce parabolically from the representation $\sigma$ twisted by $1/2$ to a representation of $\SO_{4n}(F),$ we get a reducible representation whose Langlands quotient has a generalized Shalika model. Similarly, if we induce from $\sigma$ (twisted buy $1/2$) parabolically to $\Sp_{4n}(F),$ the representation is reducible, and its Langlands quotient has a non-zero symplectic linear model. It turns out that these two Langlands quotients are related through theta correspondence. This fact fits nicely into interpretation of  symplecticity of representations of $\GL_{2n}(F)$ in terms of various functorialities and models existing on members of a dual pair $(\O_{4n}(F),\Sp_{4n}(F));$ this is nicely explained in \cite{JNQ2}, p.~541.

In this note, we answer a question of Jiang posed in \cite{JNQ2}, p.~542.  Namely, as we mentioned above, in the specific cases of induction from an irreducible supercuspidal symplectic representation of $\GL_{2n}(F),$ the corresponding Langlands quotients, which have a non-zero generalized Shalika model, and a non-zero symplectic linear model, respectively, are related through the theta correspondence. We prove that this feature occurs generally; i.e., if an irreducible smooth representation of $\SO_{4n}(F)$ has a non-zero generalized Shalika model, then its  small theta lift to  $\Sp_{4n}(F)$ is non-zero and has a non-zero symplectic linear model. This result suggests that the functorialities mentioned in the preceding paragraph (cf. \cite{JNQ2}, p.~541) can be generalized in an appropriate setting, raising further questions about Gelfand-Graev models and Fourier-Jacobi models of the representations of $\SO_{4n}(F)$ and  $\SO_{4n}(F)$ we have studied.

Our proof is based on a direct calculation of a twisted Jacquet module of the Weil representation (for a fixed additive character), and not on the more thorough study of the properties of representations having generalized Shalika or symplectic linear model.
 We adopted the latter approach  in a toy  example where we worked out the case of $n=1$ (\cite{DHL}). Here a slight disambiguation is needed (as we explain in the next subsection), since actually $\O_{4n}(F)$ and $\Sp_{4n}(F)$ occur as a dual reductive pair, so we need to extend this irreducible representation of  $\SO_{4n}(F)$ to an irreducible representation of  $\O_{4n}(F).$

\subsection{Notation and Preliminaries}

Let $F$ be a non-archimedean field of characteristic zero. We use Howe duality conjecture, which is now proved for any residual characteristic (cf.\cite{Gan_Takeda_Howe_conjecture}), so we do not need any additional assumptions on residual characteristic. We fix a non-trivial additive character $\psi: F \rightarrow \C^*$.

Let 
\[J_n:= \begin{pmatrix} & & & 1 \\ & & 1 & \\ & \Ddots & & \\ 1 & & &  \\ \end{pmatrix} \in \GL_{n}(F).\]
We realize the $F$--points of (split) special orthogonal group $\O_{4n}$ as
\[\O_{4n}(F) = \{A \in \GL_{4n}(F) | A^t J_{4n}A = J_{4n},\}\]
and $\SO_{4n}(F)$ is realized  a subgroup of $\O_{4n}(F)$ consisting of matrices of determinant  $1.$ We fix the maximal diagonal torus $T$ and the Borel subgroup $B$ of upper triangular matrices in $\SO_{4n}(F).$   We let $P=MN$ be a standard maximal parabolic subgroup of $\SO_{4n}(F)$, whose Levi subgroup $M$ is isomorphic to $\GL_{2n}(F)$.\\

 It is embedded via 
\[\iota:\GL_{2n}(F) \hookrightarrow \SO_{4n}(F) , \ g \mapsto  \begin{pmatrix} g & 0 \\ 0 & J_{2n}{g}^{-t}J_{2n} \\\end{pmatrix}\]
and the $F$-points of  its unipotent radical $N$ are given by all matrices
\[y(X) = \begin{pmatrix} I_{2n} & X \\ 0 & I_{2n} \\\end{pmatrix}, \]
such that ${X^t} = - J_{2n} X J_{2n}$. We refer to $P$ as the Siegel subgroup.
The subgroup $\mathcal{H} \subset P(F)$ generated by all $\iota(g)$ for $g \in \Sp_{2n}(F)$ and all $y \in N(F)$ is called the \textit{generalized Shalika subgroup} of $\SO_{4n}(F)$.
Here $\Sp_{2n}(F)$ is the symplectic group realized as
\[\Sp_{2n}(F) = \left \{A \in \GL_{2n}(F) | A^t\begin{bmatrix}
0&J_n\\
-J_n&0\\
\end{bmatrix}A = \begin{bmatrix}
0&J_n\\
-J_n&0\\
\end{bmatrix}\right \}.\]
We consider $\psi$ to be a character  of $N$ by 
 \[\psi(y(X)) = \psi\left(tr\left(\begin{pmatrix} -I_n&0\\
0&I_n\end{pmatrix}X\right)\right)\] 
and then we  extend it  to a character $\psi_{\mathcal{H}}$ of $\mathcal{H}$ by demanding it is trivial on $\iota(\Sp_{2n}(F))$ (this is well defined because it is easily checked that $\mathcal{H}$ is the stabilizer of a character $\psi$ in $P$). 
 
\begin{defn} 
An irreducible admissible representation $\pi$ of $\SO_{4n}(F)$ is said to have a non-zero generalized Shalika model if 
\[\mathrm{Hom}_{\mathcal{H}}(\pi, \psi_{\mathcal{H}}) \neq 0.\]
\end{defn}

The group $\Sp_{2n}(F)\times \Sp_{2n}(F)$ injects into $\Sp_{4n}(F)$ via
\begin{equation}
\label{eq_SL_Sp4}
 \left(\begin{pmatrix}  a & b \\  c&d \\ \end{pmatrix}, \begin{pmatrix}  a_1 & b_1 \\  c_1&d_1 \\ \end{pmatrix}\right)\mapsto \begin{pmatrix}  a & & &  b \\ & a_1 & b_1& \\ & c_1 & d_1 & \\ c& & & d \\ \end{pmatrix}.
 \end{equation}
Here $a,b,c,d,a_1,b_1,c_1,d_1$ are $n\times n$ matrices.

\begin{defn} An irreducible admissible representation $\pi$ on $\Sp_{4n}(F)$ has a symplectic linear model if 
\[ \mathrm{Hom}_{\Sp_{2n}(F)\times \Sp_{2n}(F)}(\pi, 1_{\Sp_{2n}(F) \times \Sp_{2n}(F)}) \neq 0.\]
\end{defn}

Since we need representations of the full orthogonal group to enter the theta correspondence, we recall the following well-known criterion. Let $\epsilon \in \O_{2n}(F)$ be the element 
$$ \epsilon=  \begin{pmatrix} I_{n-1} & & &  \\ & & 1 &  \\ &  1 & & \\  & & & I_{n-1}  \end{pmatrix}. $$
For an irreducible admissible representation $\tau$ of $\SO_{2n}(F),$ we denote by $\tau^{\epsilon}$ representation of $\SO_{2n}(F)$ on the same space, defined by $\tau^{\epsilon}(g)=\tau(\epsilon g\epsilon^{-1}).$ We can pass between irreducible admissible representations of $\O_{2n}(F)$ and $\SO_{2n}(F)$ as follows:

\begin{lem}[cf.\cite{MVW} 3.II.5, Lemme]\label{lem:thetacorres} \
\label{lem_SO_O}

\begin{enumerate}
	\item Let $\pi$ be an irreducible admissible representation of $\O_{2n}(F)$. Then \\ $\pi|_{\SO_{2n}(F)}$ is irreducible if and only if $\pi \ncong \pi \otimes det$.
	\item Let $\tau$ be an irreducible admissible representation of $\SO_{2n}(F)$. Then either 
	\begin{itemize}
	\item[(A)] $\tau \ncong \tau^\epsilon;$ then  $Ind^{\O_{2n}(F)}_{\SO_{2n}(F)}(\tau)=:\pi$ is irreducible and satisfies $\pi=\pi\otimes det$, or 
	\item[(B)] $\tau \cong \tau^\epsilon;$ then  $Ind^{\O_{2n}(F)}_{\SO_{2n}(F)}(\tau)$ is reducible and the direct sum of two non-equivalent irreducible representations $\pi$ and $\pi\otimes det$. 
	\end{itemize}
\end{enumerate}
\end{lem}

We use this lemma to adapt the theta correspondence to representations of $\SO_{2n}(F).$ Let $\omega_{m,k}$ denote the Weil representation (with respect to an additive character $\psi'$) of a dual pair consisting of the split orthogonal group $O(V)(F)$  where dimension of $V$ is $2m$ and of the symplectic group $\Sp(W)$ where the dimension of $W$ is $2k.$ The maximal quotient of $\omega_{m,k}$ on which  $O(V)(F)=\O_{2m}(F)$ acts as a multiple of an irreducible representation $\pi$ decomposes as $\pi \otimes \Theta(\pi,k),$ where $\Theta(\pi,k)$ is a finite-length $\Sp_{2k}(F)$--module. This module has  the unique irreducible quotient (Howe conjecture) which we denote by $\theta (\pi,k).$ We analogously define an irreducible $\Sp_{2k}(F)$--module $\theta(\tau,k)$ for an irreducible representation $\tau$ of  $\SO_{2m}(F).$ We have

\begin{equation}
\label{eq_theta1}
\theta(\tau, k) \cong  \theta(\pi,k) \text { if (A)}
\end{equation}
\[\theta(\tau, k) := \theta(\pi,k) \oplus \theta(\pi \otimes det,k) \text { if (B)}.\]
\begin{rem}
In the remainder of this paper, we are concerned with the theta lifts of irreducible representations of $\SO_{4n}(F)$ to irreducible representations of $\Sp_{4n}(F),$ so we are always dealing with the Weil representation $\omega_{2n,2n}$ so we denote $\theta (\tau, 2n)$ by $\theta (\tau).$ We retain the notation from Lemma \ref{lem_SO_O}.

Assume that $\tau$ is in irreducible representation of $\SO_{4n}(F)$ such that it satisfies condition $(A)$ from  Lemma \ref{lem_SO_O} and that it has a non-zero generalized Shalika functional, say $\lambda.$  Then, $\pi|_{\SO_{4n}(F)}=\tau\oplus \tau^{\epsilon}$ and we can define a Shalika functional on the representation $\pi$ by prescribing that it is equal to $\lambda$ on $\tau$ and zero on $\tau^{\epsilon}.$ If $\tau$ with  a non-zero generalized Shalika model satisfies $(B)$ then the situation is even more straightforward since then $\pi|_{\SO_{4n}(F)}=\tau.$
So, we may conclude that we can always extend generalized  Shalika functional from irreducible representation of $\SO_{4n}(F)$ to irreducible representations of $\O_{4n}(F)$ in the sense of Lemma \ref{lem_SO_O}.

We note that in  (very limited number) of explicitly known representations $\tau$ with the non-zero generalized  Shalika models (\cite{JNQ2},\cite{JQ_residues},\cite{DHL}), we always had in these examples the situation $(A).$  We know that the following holds (\cite{Sun_Zhu}):
\begin{thm}
\label{theta_occurrence}Assume that $\sigma$ is an irreducible admissible representation of the split $\O_{2m}(F).$ Then the following holds:
\[n(\sigma)+n(\sigma\otimes \mathrm{det})=2m.\]
Here $n(\sigma)$ denotes the rank of the first non-zero occurrence of the representation $\sigma$ in theta correspondence.
\end{thm}

Because of that,  if $\tau$ is irreducible representation of $\SO_{4n}(F)$ in situation $(A),$ we have that $n(\pi)=2n$ and $\theta(\tau,2n)=\theta(\pi,2n)\neq 0.$ We denote $\theta'(\tau)=\theta(\tau,2n).$

If  $\tau$ is in situation $(B),$ at least one of the representations $\pi,\;\pi \otimes det$ has a non-zero theta lift to the rank $2n.$ Now, we denote by $\theta'(\tau)$ one of the non-zero lifts $\theta(\pi,2n)$ or $\theta(\pi \otimes det,2n)$ (and both $\pi$ and $\pi \otimes det$ have a non-zero generalized Shalika model).
\end{rem}

We use $\mathrm{ind}$ to denote the compact induction, and $\mathrm{Ind}$ to denote the non-compact induction. By $\twoheadrightarrow $ we denote a surjective mapping. From now on, we study representations of groups $\SO_{4n}(F)$ and $\Sp_{4n}(F)$ for $n\ge 2,$ since $n=1$ case is resolved in \cite{DHL}.
\section{}
We continue to assume that $(\pi,V)$ is an irreducible representation of $\O_{4n}(F)$ with a non-zero generalized Shalika model such that $\theta(\pi)\neq 0.$
We want to express a property of having non-zero generalized Shalika model in terms of twisted Jacquet modules. We continue to use the notation from the previous section. We form a subspace
\[V_{\psi}(N):=\span\{\pi(n)v-\psi(n)v:v\in V,n\in N\},\]
where $N$ is the unipotent radical of the Siegel standard parabolic subgroup of $\SO_{4n}(F).$
Then, it is straightforward that  the twisted Jacquet module $R_{\mathcal{H},\psi}(\pi):=V/V_{\psi}(N)$ is a $\Sp_{2n}(F)$--module, since, by definition, $\Sp_{2n}(F)\subset \GL_{2n}\cong M$ is a stabilizer of a character $\psi$ of $N.$ Now, the existence of the non-zero generalized Shalika model on $\pi$ is equivalent to the fact that $R_{\mathcal{H},\psi}(\pi),$ as  a $\Sp_{2n}(F)$--module, has the trivial quotient, i.e. there exists a non-zero functional $\lambda$ on $R_{\mathcal{H},\psi}(\pi)$ satisfying
\[\lambda(\pi(s)v+ V_{\psi}(N))=\lambda(v+ V_{\psi}(N)).\]

\subsection{Calculation of $R_{\mathcal{H},\psi}(\omega_{2n,2n})$}
Recall that we view $\omega_{2n,2n}$ as a representation of $\O_{4n}(F)\times \Sp_{4n}(F).$ The above discussion motivates us to examine $R_{\mathcal{H},\psi}(\omega_{2n,2n})$ more thoroughly. This is obviously an $\Sp_{2n}(F)\times \Sp_{4n}(F)$-module.   Note that we have a non-trivial additive character $\psi$ appearing in the definition of the generalized Shalika model; assume that an additive character $\psi_a(x):=\psi (ax),$ where $a\in F^{\ast},$ enters the definition of theta correspondence (we do not emphasize $\psi_a$ in the notation of $\omega_{2n,2n}$). A general description of the twisted Jacquet modules of this kind is given in (\cite{MVW}, pp.~72, 73). We further elaborate on this description which is given not necessarily for the Weil representation, but in the more general context.

We study the Schroedinger model of the Weil representation $\omega_{2n,2n}$ defined in the following way: let $V=V_{2n}'\oplus V_{2n}''$ be a complete polarization of the  quadratic space $V$ on which $\O_{4n}(F)$ acts. Let $W$ be $4n$--dimensional skew-symmetric space on which $\Sp_{4n}(F)$ acts. We denote by  $\mathbf{W}=V\otimes W=V_{2n}'\otimes W\oplus V_{2n}''\otimes W.$ Then, the Schroedinger model of  $\omega_{2n,2n}$ is realized on the Schwartz space $S(V_{2n}'\otimes W).$ Sometimes we use an isomorphism $V_{2n}'\otimes W\cong W^{2n},$ so that given a basis $\{e_1,\ldots,e_{2n}\}$ of an isotropic space $V_{2n}'$ we have 
\[e_1\otimes w_1+\cdots e_{2n}\otimes w_{2n}\mapsto (w_1,\ldots,w_{2n}).\]
 To be able to directly apply formulas for the Weil representation given in (\cite{Kudla1}, p. 38) we take a  little bit different matrix realization of $\O_{4n}(F)$ (isomorphic to ours defined above) where in the definition of $\O_{4n}(F)$ the symmetric form is defined not by using the matrix $J_{4n}$ but the matrix $\begin{bmatrix}
0&I_{2n}\\
I_{2n}&0\\
\end{bmatrix}.$ Then,
\[N=\left\{n(S)=\begin{bmatrix}
I_{2n}&S\\
0&I_{2n}\\
\end{bmatrix}:S^{t}=-S\right\}.\] 
Note that then the action of $N$ in $\omega_{2n,2n}$  is given by the homothety (\cite{Kudla1}, p. 38)
\[\omega_{2n,2n}(n(S),1)\phi(w)= \psi_a( \frac{1}{2}tr(\langle w,w\rangle S))\phi(w),\]
where $w=(w_1,\ldots w_{2n})\in W^{2n},$ $\phi \in S(W^{2n}).$ Here $\langle x,x\rangle$ denotes $2n \times 2n$ skew-symmetric matrix whose $(i,j)$-entry is $\langle w_i,w_j\rangle.$
We examine (we adopt the notation of \cite{MVW}, p.~72)
\[\Omega(\psi)=\left \{w\in W^{2n}:\psi_a( \frac{1}{2}tr(\langle w,w\rangle S))=\psi_{\mathcal{H}}(S)=\psi(tr(\begin{bmatrix}
0&I_{n}\\
-I_{n}&0\\
\end{bmatrix}S))\right \},\]
where  the action of the Shalika character is adjusted because of the modified definition of $N.$ By changing $a\mapsto a^{-1},$
we get the condition
\[\Omega(\psi)=\left \{w\in W^{2n}:\psi (tr(S(\frac{1}{2}(\langle w,w\rangle-a\begin{bmatrix}
0&I_{n}\\
-I_{n}&0\\
\end{bmatrix})))=1,\forall S^{t}=-S \in M_n(F)\right \}.\]

By Lemma on p. 73 of \cite{MVW}, the restriction on $\Omega(\psi)$ gives the isomorphism of $R_{\mathcal{H},\psi}(\omega_{2n,2n})$ with the action of $\Sp_{2n}(F)\times \Sp_{4n}(F)$ on  $S(\Omega(\psi)).$ Now we examine this action more throughly.

We can get rid of $\psi$ in the above definition of $\Omega(\psi).$  We see that in the following calculation. We define an skew-symmetric matrix $A:=\frac{1}{2}\langle w,w\rangle-a\begin{bmatrix}
0&I_{n}\\
-I_{n}&0\\
\end{bmatrix}.$ We put $A=\begin{bmatrix}x&b\\
-b^t&d\\
\end{bmatrix},$ where $x^t=-x, d^t=-d$ and $S=\begin{bmatrix}a_1&b_1\\
-b_1^t&d_1\\
\end{bmatrix}$ with $a_1^t=-a_1, d_1^t=-d_1.$ The condition becomes
\[\psi(tr(a_1x-b_1b^t-b_1^tb+d_1d))=1,\]
for all $\begin{bmatrix}a_1&b_1\\
-b_1^t&d_1\\
\end{bmatrix}.$ If we take  $a_1=d_1=0$ and $b_1=\lambda e_{i,j},$ where $e_{i,j}$ is a $n\times n$ matrix whose entries are all zero except  the entry $(i,j)$ which is $1.$
We get that  $\psi(2\lambda b_{i,j})=1,$ for all $\lambda \in F^{\ast}.$ Since $\psi$ is non-trivial, we get that $b_{i,j}=0.$ This holds for every $(i,j)$ so that $b=0.$ Since $n\ge 2,$ we can take $a_1=d_1=\lambda e_{i,j}-\lambda e_{j,i},$ for some $i\neq j.$ Then, the condition becomes
$\psi(2\lambda (x+d)_{j,i})=1,\forall \lambda \in  F^{\ast}.$ We get that $ (x+d)_{j,i}=0,$ for all $j\neq i.$ We get that $x+d=0.$ If we take $a_1=-d_1=\lambda e_{i,j}-\lambda e_{j,i},$ for some $i\neq j,$ we analogously get $(x-d)_{j,i}=0$ and then we get $x=d=0.$ Thus,
\[A=\frac{1}{2}\langle w,w\rangle-a\begin{bmatrix}
0&I_{n}\\
-I_{n}&0\\
\end{bmatrix}=0.\] 
Thus,
\[\Omega(\psi)=\left \{w\in W^{2n}:\frac{1}{2}\langle w,w\rangle-a\begin{bmatrix}
0&I_{n}\\
-I_{n}&0\\
\end{bmatrix}=0\right \}.\]
Recall that the action of $\Sp_{2n}(F)\times \Sp_{4n}(F)$ on $w=e_1\otimes w_1+\cdots e_{2n}\otimes w_{2n}$ is given as follows: for $(g_1,g_2)\in \Sp_{2n}(F)\times \Sp_{4n}(F)$ we have
\[(g_1,g_2)(e_1\otimes w_1+\cdots e_{2n}\otimes w_{2n})=g_1e_1\otimes g_2 w_2+\cdots+g_1e_{2n}\otimes g_2w_{2n}.\]
We put $g_1e_i=\sum_{l=1}^{2n}a_{l,i}e_l,\;i=1,\ldots,2n.$ So we get
\begin{equation}
\label{eq:action}
(g_1,g_2)(e_1\otimes w_1+\cdots e_{2n}\otimes w_{2n})=e_1\otimes (\sum_{i=1}^{2n}a_{1,i}g_2w_i)+\cdots+e_{2n}\otimes (\sum_{i=1}^{2n}a_{2n,i}g_2w_i).
\end{equation} 
We denote $w_j'=\sum_{i=1}^{2n}a_{j,i}g_2w_i.$ Now, it is a straightforward that
\begin{equation}
\label{eq:preserving_product}
\langle w_i',w_j' \rangle=\langle w_i,w_j \rangle, \forall i,j
\end{equation}
 (we use that $g_1\in Sp_{2n}(F),$ where we now realize   $Sp_{2n}(F)$ as
\[\Sp_{2n}(F)=\left\{g_1\in GL_{2n}(F):g_1^t\begin{bmatrix}
0&I_{n}\\
-I_{n}&0\\
\end{bmatrix}g_1=\begin{bmatrix}
0&I_{n}\\
-I_{n}&0\\
\end{bmatrix}\right\}).\]
This is, of course, what we knew in advance and it just means that the action of $\Sp_{2n}(F)\times \Sp_{4n}(F)$ preserves $\Omega(\psi).$ 

We want to analyze the orbits of this action.

\begin{lem}
\label{lem:transitive}
The action of $\Sp_{2n}(F)\times \Sp_{4n}(F)$ on  $\Omega(\psi)$ is transitive.
\end{lem}
\begin{proof} Note that for $w=e_1\otimes w_1+\cdots+ e_{2n}\otimes w_{2n}=(w_1,\ldots,w_{2n})\in \Omega (\psi)$ the defining relation of $ \Omega (\psi)$ guarantees  that the set $\{w_1,w_2,\ldots,w_{2n}\}$ is linearly independent (these vectors form a symplectic basis (up to scalar) of $2n$--dimensional non-degenerate subspace of $W$). An element $g_2\in \Sp_{4n}(F)$ turns $\span\{w_1,w_2,\ldots,w_{2n}\}$ into another non-degenerate $2n$--dimensional subspace of $W$ with a (up to scalar) symplectic basis $\{g_2 w_1,\ldots g_2w_{2n}\},$ and then $g_1$  acts on the $\{g_2w_1,\ldots g_2w_{2n}\}$ by turning it into another basis of the same space.

Let $w=(w_1,\ldots,w_{2n}), w'=(w_1',\ldots,w_{2n}')\in \Omega(\psi)$ and denote 
\[V_1=\span\{w_1,\ldots,w_{2n}\}\text{ and } V_2=\span \{w_1',\ldots,w_{2n}'\}.\]
We define $f:V_1\to V_2$ with $f(w_i)=w_i',\;i=1,2,\ldots,2n.$ It is obvious that $f$ is an isometry. By the  Witt's theorem, there exists an isometry on $W$ (thus an element $g_2 \in\Sp_{4n}(F)$) extending $f.$ This means that $(1,g_2)w=w'.$

 \end{proof} 
We fix $w_0=(w_1,\ldots,w_{2n})$ in $\Omega(\psi)$ and let $G_1\subset \Sp_{2n}(F)\times \Sp_{4n}(F)$ is a stabilizer of that point.  By the known results (cf.~\cite{MVW}, p.73), since there is only one orbit for this action on $\Omega(\psi),$ we have
\begin{equation}
\label{eq:the representation_I}
R_{\mathcal{H},\psi}(\omega_{2n,2n})\cong \ind_{G_1}^{\Sp_{2n}(F)\times \Sp_{4n}(F)}\omega_{w_0}.	
\end{equation}

Here $\omega_{w_0}$  is a representation of $\Sp_{2n}(F)\times \Sp_{4n}(F)$ satisfying
\[(\omega_{2n,2n})(g_1,g_2)f(w_0)=\omega_{w_0}(g_1,g_2)f(w_0(g_1,g_2)).\]
Since $\omega_{w_0}$ is a character, it must be equal to 1. Indeed, when we check the formulas from (\cite{Kudla1}, p.~38) we get
\[(\omega_{2n,2n})(g_1,1)f(w_0)=f(g_1^te_1\otimes w_1+\cdots+g_2^te_{2n}\otimes w_{2n}),\]
and
\[(\omega_{2n,2n})(1,g_2)f(w_0)=f(e_1\otimes g_2^{-1}w_1+\cdots +e_{2n}\otimes g_2^{-1}w_{2n}).\]

\begin{lem}Let $G_1$ be the stabilizer of $w_0$ with respect to  $\Sp_{2n}(F)\times \Sp_{4n}(F)$ action given by (\ref{eq:action}). Then, 
\[G_1\cong \Sp_{2n}(F)\times \Sp_{2n}(F)\]
given with
\[(g_1,g_2)\mapsto (g_1^{-t},(g_1,g_2)),\]
where $(g_1,g_2)$ from the right hand side belongs to $\Sp_{2n}(F)\times\Sp_{2n}(F)\subset \Sp_{4n}(F),$ and where $W$ is decomposed as a orthogonal direct sum of non-degenerate symplectic spaces of dimensions $2n$ and each copy of $\Sp_{2n}(F)$ is the symplectic group of the corresponding subspace.
\end{lem}

\begin{proof} According to the interpretation of this action given in the proof of Lemma \ref{lem:transitive}, for $(g_1,g_2)\in \Sp_{2n}(F)\times \Sp_{4n}(F)$ to be in $G_1,$ it is needed that, for $V_1:=\span\{w_1,\ldots,w_{2n}\},$ we have $g_2(V_1)=V_1.$ Since $V_1$ is non degenerate, we have the orthogonal direct decomposition
\[W=V_1\oplus V_1^{\perp},\]
where $V_1^{\perp}$ denotes the orthogonal complement of $V_1.$ Now, we immediately have $g_2(V_1^{\perp})=V_1^{\perp}$ and $g_2\mapsto (g_2|_{V_1},g_2|_{V_1^{\perp}})$ is injective. Note that $g_2|_{V_1}$ and $g_2|_{V_1^{\perp}}$ belong to the symplectic groups of $V_1$ and $V_1^{\perp},$ respectively. Then, for $g_1$ such that $(g_1,g_2)\in G_1$ we must have (from (\ref{eq:action})) that $g_1=(g_2|_{V_1})^{-t}.$
\end{proof}

Note that a function $f$  from $\ind_{G_1}^{\Sp_{2n}(F)\times \Sp_{4n}(F)} 1=\ind_{\Sp_{2n}(F)\times \Sp_{2n}(F)}^{\Sp_{2n}(F)\times \Sp_{4n}(F)} 1$ satisfies 
\[f(g_1'^{-t},(g_1',g_2'))(\alpha,\beta))=f((\alpha,\beta)),\]
for all $(\alpha,\beta)\in \Sp_{2n}(F)\times \Sp_{4n}(F)$ and $(g_1'^{-t},(g_1',g_2'))\in \Sp_{2n}(F)\times \Sp_{4n}(F)).$ $f$ is also smooth and compactly supported in $\Sp_{2n}(F)\times \Sp_{4n}(F)$ modulo $G_1.$
 Note that this means that $f((\alpha,\beta))=f(1,(\alpha^t,1)\beta)),$ so that $f$ is completely determined by  its restriction to $\Sp_{4n}(F).$ We define 
 \[\phi_f(\beta)=f(1,\beta).\]
 We also note that $\phi_f:\Sp_{4n}(F)\to \C$ is left $\Sp_{2n}(F)$-- invariant with respect to the second copy of  $\Sp_{2n}(F).$
 We get that 
 \[f\mapsto \phi_f\]
 is a bijection from $\ind_{\Sp_{2n}(F)\times \Sp_{2n}(F)}^{\Sp_{2n}(F)\times \Sp_{4n}(F)} 1$ to $\ind_{\Sp_{2n}(F)}^{\Sp_{4n}(F)} 1$ (we easily get that $\phi_f$ is smooth and compactly supported modulo the second copy of $\Sp_{2n}(F)$). The action of $\Sp_{2n}(F)\times \Sp_{4n}(F)$ on $\ind_{\Sp_{2n}(F)\times \Sp_{2n}(F)}^{\Sp_{2n}(F)\times \Sp_{4n}(F)} 1$ becomes 
 \begin{equation}
 \label{group_action}
 R(g_1,g_2)\phi(\beta)=\phi((g_1^t,1)\beta g_2)
 \end{equation}
 on $\ind_{\Sp_{2n}(F)}^{\Sp_{4n}(F)} 1.$ We have proved

 \begin{prop}$R_{\mathcal{H},\psi}(\omega_{2n,2n})$ is, as a $\Sp_{2n}(F)\times \Sp_{4n}(F)$ module, isomorphic to $\ind_{\Sp_{2n}(F)}^{\Sp_{4n}(F)} 1$ with the action of $\Sp_{2n}(F)\times \Sp_{4n}(F)$ given by (\ref{group_action}).
 \end{prop}

 Note that the first copy of $\Sp_{2n}(F)$ acts as the left translation; we denote this action by $\lambda.$ 

 Now we want to analyze the biggest quotient of $\ind_{\Sp_{2n}(F)}^{\Sp_{4n}(F)} 1$  on which $\Sp_{2n}(F)$ (through $\lambda$) acts trivially. To that end, we define
  \[S'=\span\{\lambda (g_2)\phi-\phi:g_2\in \Sp_{2n}(F),\phi \in \ind_{\Sp_{2n}(F)}^{\Sp_{4n}(F)} 1\}.\]
  Obviously, $\ind_{\Sp_{2n}(F)}^{\Sp_{4n}(F)} 1/S'$ is that quotient; we consider it as a  $\Sp_{4n}(F)$--module.

 \begin{thm}
 \label{the_largest_quo_in_Weil}
 There is an isomorphism of  $\Sp_{4n}(F)$--modules:
 \[\ind_{\Sp_{2n}(F)}^{\Sp_{4n}(F)} 1/S'\cong \ind_{\Sp_{2n}(F)\times \Sp_{2n}(F)}^{\Sp_{4n}(F)} 1.\]
 \end{thm}
\begin{proof}We denote
\[T(\phi)(g)=\int_{\Sp_{2n}(F)}\phi((x,1)g)\deri x.\]
For $\phi\in \ind_{\Sp_{2n}(F)}^{\Sp_{4n}(F)} 1$ the integral on the right hand side converges. Indeed, fix $g\in \Sp_{4n}(F).$ We know that there exist a compact set $C_1\subset \Sp_{4n}(F)$ such that $\mathrm{supp} \phi\subset (\{1\}\times \Sp_{2n}(F))C_1.$ Assume that $\phi((x,1)g)\neq 0,$ which means that $(x,1)\in (\{1\}\times \Sp_{2n}(F))C_1g^{-1}.$ We denote $C_1':=C_1g^{-1}.$ Note that $C_1'\cap \Sp_{2n}(F)\times \Sp_{2n}(F)$ is a compact set in $ \Sp_{2n}(F)\times \Sp_{2n}(F).$ We denote by $p_i,\;i=1,2$ the projections from $ \Sp_{2n}(F)\times \Sp_{2n}(F)$ to the first and the second copy of $\Sp_{2n}(F).$
This means that 
\[(x,1)\in (\{1\}\times \Sp_{2n}(F))(p_1(C_1')\times p_2(C_1'))=p_1(C_1')\times  \Sp_{2n}(F).\]
This means that $x\in p_1(C_1'),$ which is a compact set in (the first copy of) $\Sp_{2n}(F).$ Thus, $x\mapsto \phi ((x,1)g)$ is a smooth function with the compact support in $\Sp_{2n}(F).$ Thus, $T(\phi)$ is well defined function on $\Sp_{4n}(F)$. Also, it is smooth. Again, if $C_1$ denotes the compact set  in $\Sp_{4n}(F)$ related to the support of $\phi$ as above, then it is easy to see that $\mathrm{supp} T(\phi)\subset (\Sp_{2n}(F)\times \Sp_{2n}(F))C_1.$
Also, it is immediate that the following holds
\[T(\phi)((g_1,g_2)g)=T(\phi)(g),\forall (g_1,g_2)\in Sp_{2n}(F)\times Sp_{2n}(F), g\in Sp_{4n}(F),\]
and 
\[T(R(g)\phi)=R(g)T(\phi).\]
Thus, $T$ is $Sp_{4n}(F)$--intertwining operator between $\ind_{\Sp_{2n}(F)}^{\Sp_{4n}(F)} 1$ and $\ind_{\Sp_{2n}(F)\times \Sp_{2n}(F)}^{\Sp_{4n}(F)} 1.$
We immediately see that $T|_{S'}=0.$ 

We now prove the surjectivity of the operator $T.$ We use (\cite{Bourbaki_Haar}, cf.~\cite{Casselman_notes}, p.~27) to introduce the mapping
\[P_{\delta_1}:C_c^{\infty}(\Sp_{4n}(F))\to \ind_{\Sp_{2n}(F)}^{\Sp_{4n}(F)} 1\]
given by
\[P_{\delta_1}(f)(g)=\int_{\Sp_{2n}(F)}f((1,x)g)\deri x.\]
It is known that $P_{\delta_1}$ is surjective (\cite{Casselman_notes}, p.~27).
Analogously we define  a (surjective) mapping
\[P_{\delta_2}:C_c^{\infty}(\Sp_{4n}(F))\to \ind_{\Sp_{2n}(F)\times \Sp_{2n}(F)}^{\Sp_{4n}(F)} 1\]
given by 
\[P_{\delta_2}(f)(g)=\int_{\Sp_{2n}(F)\times \Sp_{2n}(F)}f((x,y)g)\deri x \deri y.\]
We immediately see that 
\begin{equation}
\label{eq_delta1_delta2}
P_{\delta_2}(f)(g)=\int_{\Sp_{2n}(F)}P_{\delta_1}(\lambda(x^t)f)(g)\deri x=\int_{\Sp_{2n}(F)}\lambda(x^t)P_{\delta_1}(f)(g)\deri x=T(P_{\delta_1}(f))(g).
\end{equation}
Thus, $P_{\delta_2}(f)=T(P_{\delta_1}(f))$ and $T$ is surjective.

Now we prove that $\mathrm{Ker}\,T=S'.$ Assume that $\phi\in \mathrm{Ker}\,T.$ Then, there exists $f\in C_c^{\infty}(\Sp_{4n}(F))$ such that $\phi=P_{\delta_1}(f).$ Thus, $T(\phi)=P_{\delta_2}(f)=0.$ There exist an open compact subgroup $K$ of $\Sp_{4n}(F),\;g_1,\ldots,g_m\in \Sp_{4n}(F)$  and $c_1,\ldots,c_m\in \C$ such  that
\[f=\sum_{i=1}^m c_i\chi_{Kg_i}.\]
Here we assume that for $i\neq j$ $Kg_i\cap Kg_j=\emptyset$ and $\chi_{Kg_i}$ denotes the characteristic function on the right coset $Kg_i.$
We examine the first equation in (\ref{eq_delta1_delta2}). The integrating function,
 \[x\mapsto P_{\delta_1}((x,1)g)=\sum_{i=1}^mc_i\mu_{\{1\}\times \Sp_{2n}(F)}((x^{-1},1)Kg_ig^{-1}\cap \{1\}\times \Sp_{2n}(F))\]
 is locally (uniformly) constant.  Here $\mu_{\{1\}\times \Sp_{2n}(F)}$ denotes a Haar measure on $\{1\}\times \Sp_{2n}(F).$ Indeed, if we denote by $K_0:=K\cap \Sp_{2n}(F)\times \{1\},$ which is compact and open in $\Sp_{2n}(F)\times \{1\},$ we see that the function
  \[x\mapsto\sum_{i=1}^mc_i\mu_{\{1\}\times \Sp_{2n}(F)}((x^{-1},1)Kg_ig^{-1}\cap \{1\}\times \Sp_{2n}(F))\] is a constant on cosets $K_0\setminus \Sp_{2n}(F)\times \{1\}.$ Also, we effectively integrate  in  (\ref{eq_delta1_delta2}) over a compact set. We  integrate over a finite set of different cosets of $K_0\setminus \Sp_{2n}(F)\times \{1\}.$ Thus, there exist $x_1,\ldots,x_l \in \Sp_{2n}(F)\times \{1\}$ such that
\[0=\sum_{j=1}^l\int_{K_0x_j}(\lambda(x^t)P_{\delta_1}(f))(g)\deri x=\mu_{\{1\}\times \Sp_{2n}(F)}(K_0)\sum_{j=1}^lP_{\delta_1}(f)((x_j,1)g),\]
for every $g\in \Sp_{4n}(F).$
This means
\[\lambda(x_1^t)P_{\delta_1}(f)=-\sum_{j=2}^l\lambda(x_j^t)P_{\delta_1}(f),\]
so that
\[P_{\delta_1}(f)=-\sum_{j=2}^l\lambda(x_1^{-t}x_j^t)P_{\delta_1}(f).\]
This means
\[P_{\delta_1}(f)=\phi=-\frac{1}{l}\sum_{j=2}^l(\lambda(x_1^{-t}x_j^t)\phi-\phi),\]
and this means that $\mathrm{Ker}\,T=S'.$
\end{proof}

\subsection{Conclusion}
We continue to assume that $\pi$ is an irreducible representation of $\O_{4n}(F)$ with a non-zero Shalika model such that $\theta(\pi)\neq 0$ is its (irreducible) small theta lift.
We thus have
\[\omega_{2n,2n}\twoheadrightarrow \pi\otimes \theta(\pi),\]
and, since taking a twisted Jacquet module is exact, we have
\[R_{\mathcal{H},\psi}(\omega_{2n,2n})\twoheadrightarrow R_{\mathcal{H},\psi}(\pi)\otimes \theta(\pi).\]
Since we assumed that $\pi$ has a non-zero Shalika model, there is a surjective $\Sp_{2n}(F)\times\Sp_{4n}(F)$ intertwining
\[R_{\mathcal{H},\psi}(\omega_{2n,2n})\twoheadrightarrow 1_{\Sp_{2n}(F)}\otimes \theta(\pi).\]
From Theorem \ref{the_largest_quo_in_Weil} it follows that there is an epimorphism
\[\ind_{\Sp_{2n}(F)\times \Sp_{2n}(F)}^{\Sp_{4n}(F)} 1\twoheadrightarrow \theta(\pi).\]
Taking the smooth adjoint of an epimorphism above, we get that
\[\mathrm{Hom}(\widetilde{\theta(\pi)}, \Ind_{\Sp_{2n}(F)\times \Sp_{2n}(F)}^{\Sp_{4n}(F)} 1)\neq 0,\]
since $\widetilde{\ind_{\Sp_{2n}(F)\times \Sp_{2n}(F)}^{\Sp_{4n}(F)} 1}\cong  \Ind_{\Sp_{2n}(F)\times \Sp_{2n}(F)}^{\Sp_{4n}(F)} 1 .$ This is equivalent to the fact that the representation $\widetilde{\theta(\pi)}$ of $\Sp_{4n}(F)$ has a non-zero symplectic linear model. But if $\widetilde{\theta(\pi)}$ has this model, the representation $\theta(\pi)$ also has it  (cf. the proof of Theorem 17 of \cite{Ginzburg_Rallis_Soudry_explicit}) and we have proved  the following theorem.
\begin{thm}
\label{main}
Assume $\tau$ is an irreducible smooth representation of $\SO_{4n}(F)$ having a non-zero generalized Shalika model. Then, the irreducible non-zero representation $\theta'(\tau)$ (the small theta lift of $\tau,$ as explained in Introduction) has a non-zero symplectic linear model.
\end{thm}

 \smallskip
\noindent\textit{Acknowledgments.}
We wish to thank American Institute of Mathematics in Palo Alto (Automorphic forms and harmonic analysis on covering groups 2013.) and CIRM, Luminy (WINE conference 2013.) where the author began to study problems related to theta correspondence and generalized Shalika models.

This work has been supported in part by Croatian Science Foundation under the
project 9364.

\bibliographystyle{siam}
\bibliography{toy_example}

\end{document}